\documentclass [11pt]{article}

\usepackage[]{graphicx,epsf,epsfig,amsmath,amssymb,amsfonts,latexsym,amsthm}
\usepackage{cleveref}
\usepackage{color}
\usepackage{url}

\newtheorem{theorem}{Theorem}[section]

\newtheorem{lem}[theorem]{Lemma}
\newtheorem{example}[theorem]{Example}

\newtheorem{definition}{Definition}
\newtheorem{conjecture}[theorem]{Conjecture}

\newtheorem{question}[theorem]{Question}

\renewcommand{\epsilon}{\varepsilon}

\newtheoremstyle{upright}%
        {8pt plus2pt minus4pt}%
        {8pt plus2pt minus4pt}%
        {\upshape}%
        {}%
        {\bfseries\scshape}%
        {:}%
        {1em}%
        {}%
\theoremstyle{upright}
\newtheorem{remark}[theorem]{Remark}

\newcommand{\ignore}[1]{}

\renewcommand{\P}{\mathbb{P}}
\newcommand{\E}{\mathbb{E}}
\newcommand{\R}{\mathbb{R}}
\newcommand{\Z}{\mathbb{Z}}
\newcommand{\N}{\mathbb{N}}

\newcommand{\brmul}{\discretionary{\mbox{$\,\cdot$}}{}{}}
\def\reff{R_{\rm eff}}

\begin{document}

\title{Random Walk in Changing Environment}

\author{{Gideon Amir \thanks{Bar-Ilan University, Ramat-Gan, 5290002,
Israel. Email: gidi.amir@gmail.com}} \quad {Itai Benjamini
\thanks{Weizmann Institute, Rehovot, 76100,
Israel. Email: itai.benjamini@weizmann.ac.il}} \quad {Ori
Gurel-Gurevich \thanks{The Hebrew University of Jerusalem, 91904, Israel.
Email: origurel@math.huji.ac.il}} \quad {Gady Kozma
\thanks{Weizmann Institute, Rehovot, 76100, Israel. Email:
gady.kozma@weizmann.ac.il}}}

\maketitle

\begin{abstract}
In this paper we introduce the notion of \emph{Random Walk in
Changing Environment} -
a random walk in which each step is performed in a different graph
on the same set of vertices, or more generally, a weighted random walk on the same vertex and edge sets but with different (possibly 0) weights in each step. This is a very wide class of RW, which includes some well known types of RW as special cases (e.g. reinforced RW, true SAW). We define and explore various possible properties of such walks, and provide criteria for recurrence and transience when the underlying graph is $\N$ or a tree.
We provide an example of such a process on $\Z^2$ where conductances can only change from $1$ to $2$ (once for each edge) but nevertheless the walk is transient, and conjecture that such behaviour cannot happen when the weights are chosen in advance, that is, do not depend on the location of the RW.
\end{abstract}

\section{Introduction}

Theseus is thrown into Daedalus' labyrinth, this time without a ball
of thread. Noticing that the labyrinth is a subgraph of $\Z^2$,
Theseus decides to simply random walk his way out - he knows that
he will almost surely reach the exit eventually. What Theseus
doesn't know is that Daedalus, aware of the recurrency of his
labyrinth, is working relentlessly to amend this vulnerability. He is
continually digging new passages throughout the
labyrinth, following a carefully laid plan. He cannot, however,
block existing passages, only create new ones and only between
adjacent rooms, such that the labyrinth is a subgraph of $\Z^2$ at
any point. Will Theseus find his way to the exit or is it possible that
Daedalus' cunning plan will deceive him forever (with positive
probability)?

It turns out that if Daedalus is aware of Theseus whereabouts he
can devise a plan to lure poor Theseus further and further into the
labyrinth with positive probability (see Thm~\ref{thm:MAW}).
We conjecture that this is not the case if Daedalus is not aware of Theseus whereabouts.

\begin{conjecture}
If Daedalus is oblivious of Theseus location, then Theseus will
almost surely reach the exit (infinitely many times, if he chooses
to stay in the labyrinth). In other words, Theseus' Random Walk is
recurrent.
\end{conjecture}
See Section~\ref{sec:open} for formal statement and further open problems.

In this paper we introduce the notion of \emph{Random Walk in
Changing Environment} or \emph{RWCE}. Generally speaking, a RWCE is
a random walk in which each step is performed in a different graph
on the same set of vertices. By different we may mean that the set
of edges is different, but it is easier and more general to assume
that the underlying graph is the same throughout the walk and what
changes are the conductances of the edges. This is a very wide class
of RW, which includes some well known types of RW as special cases,
most notably the reinforced RW. We define and explore various
possible properties of RWCE with the conclusion that the interesting
case is when the walk is monotone (Daedalus can only create passages) and
bounded (all the passages are edges of $\Z^2$). Under these assumptions we give criteria for recurrence and transience when the underlying graph is $\N$ or a tree, even when the sequence of graphs may depend on the history of the walk (the latter we call "adaptive", see section \ref{defs} for exact definitions). \\
We show that the above criteria cannot hold for general graphs: we provide an example where the underlying graph is  $\Z^2$, but the RWCE is transient. This example is an RWCE on $\Z^2$ where each edge is started with weight $1$ and at each stage we change the weight of the edge to the right of the walker to $2$ if this was not already done. The idea behind this construction is to try and mimic the behaviour of excited random walk on two dimensions (see \cite{BENJAMINI2003EXCITED}) in which the walker gets a bias to the right whenever it visits a point for the first time, and was shown in \cite{BENJAMINI2003EXCITED} to be transient. However, it turns out that the proof carried out in \cite{BENJAMINI2003EXCITED} depends quite delicately on the model, and one must take care when working out the details. In particular, a similar attempt to mimic multi-excited RW in dimension $1$ (see \cite{zerner2005multi}) cannot succeed, as follows from our results on RWCEs on trees. Note that the above example was an adaptive RWCE, and we conjecture that such behaviour cannot occur in nonadaptive RWCE.

\textbf{Related works:} In recent years there have been a number of papers that studied related models. These works have some overlap with our model and some of the examples, but not with the results of this paper, and generally speaking the emphasis of these works are in different directions.
Avin, Koucký and Lotker \cite{avin2008explore} studied RWCE on a sequence of finite unweighted graphs (which they called "evolving graphs")
 They were interested mainly in the problem of the cover time of the walk, showing , in particular, that contrary to a regular random walk on the graph, the RWCE may have exponential cover time.
Dembo, Huang and Sidoravicius studied models on random walks on "monotone domains" - that is they assumed that the sequence of graphs in the RWCE is an increasing sequence of subgraphs of a pre-given graph, with a focus on $\Z^d$. They proved criteria for recurrence and transience of such walks, with one paper \cite{dembo2014walking} focusing on the nonadaptive case (where they also consider a continuous analog for brownian motion) , and the other \cite{dembo2014monotone} focusing on the adaptive case - that is when there is an interaction between the walk and the graph sequence. As will be seen in section \ref{examples}, the generality of these models implies that some further assumptions must be taken in order to get meaningful criteria, and Dembo, Huang and Sidoravicius focus on several interaction mechanisms (such as, e.g. the walker uncovering new edges when it approaches them) and give criteria for transience and recurrence as well as some conjectures, some of which carry a similar flavour to the ones in this paper.

\textbf{Structure of the paper:} In Section \ref{defs} we
give the basic definitions and examples of known random walk models
which falls into our framework. In Section \ref{examples} we give
simple examples of Random Walks in Changing Environment which
illustrate that if the environment is unbounded or nonmonotone then
the random walk can have (almost) any behavior. Sections \ref{main_N}
and \ref{main_T} give the main results about bounded monotone
RWCE on $\N$ and on trees, respectively. Section
\ref{counterexample} gives an example of a bounded monotone
(adaptive) RWCE on $\Z^2$ which is transient, thus showing that the
results on recurrent trees cannot be extended to general recurrent
graphs. We conclude with a conjecture and some open problems.
\section{Definitions}\label{defs}
We begin by giving a rigorous definition of what a random walk in
changing environment is, in the broadest sense. In this paper, we use discrete time, see section \ref{sec:open} for a brief discussion of the continuous time version.

\begin{definition}
A \emph{Random Walk in Changing Environment (RWCE)}, on a graph
$G=(V,E)$ is a stochastic process $\{\langle X_t,G_t \rangle \}_{t=0}^\infty$, where
$G_t=(V,E,C_t)$ are graphs with a conductances function $C_t:E\rightarrow [0,\infty)$ over a fixed
vertex set $V$ and edge set $E$, and for all $t$, $X_t \in V$ and

$$\P(X_{t+1}=v \ | \ \langle X_0,G_0 \rangle ,..., \langle X_t,G_t \rangle )=
\frac{C_t(X_t,v)}{\sum_{\{e \in E |  X_t \in e \} } C_t(e)} \ .$$

We call the sequence $\{X_t\}_{t=0}^\infty$ the \emph{Random Walk}
and the sequence $\{G_t\}_{t=0}^\infty$ the \emph{Environment}.
\end{definition}

In other words, the law of the process governs the changes in the
environment, while the distribution of $X_t$, given the history, only depends only $G_t$ and is the same as a random walk step on the graph $G$ with weights $C_t$. Note that the conductances $C_t$ may depend on the history of the process so far and on extra randomness.

In our Labyrinth example, Daedalus was creating new edges, not
changing conductances. It is easy to see, however, that the
definition using conductances is a generalization of this scenario.

\begin{definition}
A RWCE is called \emph{proper} if $0 < C_t(e) < \infty$ for all $t
\in \N$ and $e \in E$. It is called \emph{improper} otherwise.
\end{definition}

\begin{definition}
A RWCE is said to be \emph{bounded from above (below) by
$G=(V,E,C)$} if $C_t(e)\le C(e)$ ($C_t(e)\ge C$) for all $e \in E$ and all $t
\in \N$,
almost surely.
\end{definition}

Note that a RWCE bounded from above and below is necessarily proper.
All the RWCE in this paper are proper unless otherwise noted. Also
note the requirement $C_t(e)<\infty$ in the definition is formally
redundant as the conductances were defined to be real numbers.
However, in the more naive approach of a changing graph, a conductance of
infinity would correspond to merging (shorting) two vertices together.

\begin{definition}
A RWCE is called \emph{nonadaptive} if the distribution of $G_{t+1}$
given $G_0,..,G_t$ is independent of $X_0,..,X_t$. It is called
\emph{adaptive} otherwise.
\end{definition}

The Labyrinth example is nonadaptive if Daedalus is oblivious of
Theseus whereabouts or adaptive if Daedalus responds to it.

\begin{definition}
A RWCE is called \emph{monotone increasing (decreasing)} if
$C_{t+1}\ge C_t$ ($C_{t+1}\le C_t$) almost surely.
\end{definition}

The Labyrinth example is monotone increasing, since Daedalus
only adds new edges, i.e. raises the conductance by 1.

Note that the definition of a general RWCE is very broad.
Actually, it is too broad, as an adaptive, improper, nonmonotone
RWCE on the full graph can implement any behavior at any stage.
But even with some restrictions, many interesting walks can be
implemented as RWCE in a natural way. We next give several
examples of well-known random walks and how they fit into our
definition:

\begin{example}
The once-reinforced random walk (see \cite{durrett2002once}) on $\Z^2$, is a proper, adaptive
monotone increasing RWCE. At the beginning the conductance of each edge is 1,
and at each stage, if the RW traversed an edge with conductance $1$, replace it with an edge of
conductance $c$ (for a fixed constant $c$). This RWCE is bounded
between $1$ and $c$. Other reinforced random walks also fit similarly into the RWCE framework. See \cite{pemantle2007survey} for a survey of such models.
\end{example}

\begin{example}
The Bridge Burning Random Walk (where the conductance of each edge the walk traverses is reduced to $0$) on $\Z^2$, is an improper, adaptive
monotone decreasing RWCE. It is as the once-reinforced RW with $c=0$.
\end{example}

\begin{example}
The Laplacian random walk from between $v_0$ and $u$ (which is equivalent to the loop erased random walk from $v_0$ to $u$, see \cite{lawler1979self}), which starts at $X_0=v_0$ and chooses which neighbour to move to at each step with probabilities proportional to the value of the harmonic function which is $0$ on the path of the RW up to this time and $1$ on $u$, can be described as a monotone bounded improper adaptive RWCE.
\end{example}

\begin{example}
The "true" self-avoiding walk with bond repulsion (see e.g. \cite{toth1995true}) is a nearest neighbor random walk, for which the probability of jumping along an edge $e$ is proportional to $e^{-ck(e)}$, where $k(e)$ is the number of times $e$ has been traversed. This is an adaptive, monotone, proper RWCE.
\end{example}

The main question we will be interested in, is whether a given RWCE
is recurrent. Note that for RWCE, the dichotomy
between recurrence and transience is not always as clear cut as for
simple RW. There might be a difference between a.s. returning to the
origin, a.s. visiting every vertex, a.s. returning to the origin
infinitely many times and a.s. visiting every vertex infinitely many
times. Also, since no 0-1 law holds in general for RWCE, we
can have a RWCE which return to the origin infinitely many times
with probability which is positive but less then 1.

In most natural cases, however, the various possible definitions of
recurrence and transience for RWCE coincide. We will therefore use
the strictest definitions.

\begin{definition}
A RWCE on $G=(V,E)$ is called \emph{recurrent} if it
visits every vertex in $V$ infinitely many times almost surely.
A RWCE is called \emph{transient} if it visits every
vertex a finite number of times almost surely.
The RWCE is said to be of \emph{mixed type} otherwise.
\end{definition}

\section{Simple examples} \label{examples}

The aim of this section is to demonstrate the myriad possible behaviors of unrestricted
RWCE. We begin with a simple example on general graphs. Let $G$ be
any graph and $X_0$ be a vertex in $G$.


\begin{example}
For any distribution on paths in $G$ (starting with $X_0$),
there is an improper, adaptive, nonmonotone RWCE inducing this distribution on $X$.
\end{example}

Since we have complete control over the conductances of the edges emerging
from $X_t$, we can arbitrarily determine the distribution of the next
step, and therefore the distribution of the sequence.

A distribution on paths in $G$ is called \emph{elliptic} if for every
finite path in $G$, $v_0,\ldots,v_n$, with $v_0=X_0$, we have $\P(X_0=v_0,\ldots,X_n=v_n)>0$.

\begin{example}
For any elliptic distribution on paths in $G$ (starting with $X_0$),
there is a proper, adaptive, nonmonotone RWCE inducing this distribution on $X$.
\end{example}

This example is the same as the previous one except you can't have
probability 0 for any transition. Next, note that since multiplying
the conductances by some constant does not change the next step
distribution, the previous example can be made monotone, either
increasing or decreasing. Also, the starting set of conductances
$C_0$ can be arbitrary (except for conductances of edges emerging from $X_0$)
and by monotonicity the RWCE is bounded (from above or below) by
$C_0$. Put together we have:

\begin{example}
For any elliptic distribution on paths in $G=(V,E)$ (starting with $X_0$), and any (proper) choice of conductances $C$
there is a proper, adaptive, monotone (increasing or decreasing) RWCE, bounded (from below or above, resp.) by $(V,E,C)$,
 inducing this distribution on $X$.
\end{example}

If we drop monotonicity, but require boundedness instead then we can
still produce any distribution that has \emph{bounded conditional
probabilities}, i.e. the probability for traversing a given edge is uniformly
bounded away from zero. In particular, we have the following example
on $\N$:

\begin{example}
The RWCE with conductances $C_t(X_t,X_t +1)=2$ and $C_t(j,j+1)=1$ for $j\neq X_t$ is
bounded from above and below by a recurrent graph, adaptive, nonmonotone and transient.
\end{example}

Indeed, $X_t$ is simply a biased RW and is therefore transient.

We have thus seen that neither boundedness nor monotonicity are
enough to draw any significant conclusions about the RWCE, at least
in the adaptive setting.

The next example shows that even in the nonadaptive setting,
boundedness does not imply recurrence or transience.

\begin{example} \label{ex:wave}
The RWCE with conductances $C_t(j,j+1)=100$ for $t\equiv j \mod 100$ and $C_t(j,j+1)=1$
otherwise is bounded from above and below by a recurrent graph, nonadaptive, nonmonotone and transient.
\end{example}

\begin{proof}[Sketch of proof]
When $X_t\equiv t \mod 100$ the conductance to the right of $X_t$ is 100
while to the left it is only 1. Therefore, with probability 100/101,
$X_{t+1}=X_t+1$, in which case $X_{t+1}=t+1(\mod 100)$. This happens
for an expected number of 101 times, after which the walk is simple
until the next 100 conductance "catches up". This takes about 100
steps in which the expected displacement is 0. Bipartiteness of the graph ensures the walk never gets a bias to the left. All in all, the RW
gets a strong bias to the right about half the time and so it is
transient.
\end{proof}

Note that the same conductances would work even if the RW had
some probability of staying at the same vertex, thus nullifying the
bipartiteness of the graph, though the calculation would be slightly more involved. The reason being that while the walker would sometime get a bias to the left, the wave would "pass" the walker once a step to the left was made.

Similarly, we can make the RW recurrent, even if it is bounded by a
transient graph.

\begin{example}
The RWCE with conductances $C_t(j,j+1)=1000 \cdot 2^j$ for $t\equiv -j \mod 100$ and $C_t(j,j+1)=2^j$
otherwise is bounded from above and below by a transient graph, nonadaptive, nonmonotone and recurrent.
\end{example}

\begin{proof}[Sketch of proof]
The argument is the same as in example \ref{ex:wave}, except that when $X_t\not\equiv -t \mod 100$ the RW gets a $(1/3,2/3)$ bias to the
right instead of being balanced. However, simple calculation shows that this bias is not enough to counter the bias to the left when
$X_t\equiv -t \mod 100$, which once caught would persist as long as the walker keeps going left.
\end{proof}

%
%
%
%

\section{RWCE on $\N$} \label{main_N}

In this section we study RWCE whose underlying graph is $\N$ (with
edges between consecutive integers). All the theorems here apply
equally to RWCE on $\Z$, but the proofs are slightly simpler for
$\N$ since there's only one way to infinity. For such graphs we can prove
quite general conditions which ensure the RWCE is recurrent (or
transient).

The main idea of the proofs in this section and the next is as follows.
We will define a \emph{potential sequence} - an adaptive sequence
of functions $F_t:V\rightarrow \R^+$ satisfying:

\begin{enumerate}
\item \emph{Harmonicity:} $F_t$ is harmonic on $(V,E,C_t)$ except at 0.
\item \emph{Monotonicity:} $F_t(v)$ is either monotone increasing for all $v\in V$ or monotone decreasing for all $v\in V$.
\end{enumerate}

Note that $F_t$ may depend on $H_t$, the history of the RWCE up to time $t$,
even if the RWCE itself is nonadaptive. The two properties above
imply that $F_t(X_t)$ is either a supermartingale or a submartingale as long as $X_t\neq 0$.
This is because $\E(F_t(X_{t+1})|H_t)=F_t(X_t)$ by harmonicity of
$F_t$ and because $F_{t+1}(X_{t+1})\geq F_t(X_{t+1})$ (or $\leq$) by
monotonicity. We will then use the optional stopping theorem to deduce bounds on the
probability of return to 0. Note that related ideas were used by Vervoort \cite{vervoort2002reinforced} and even earlier by Davis \cite{davis1990reinforced} in the context of reinforced random walks.

The following theorems all require the RWCE be bounded from below
and above by some graph. When this condition holds, the walk is
\emph{elliptic} (uniformly in time), that is, the probability of traversing each edge when
the walk is at one of its endpoints is bounded away from 0. On $\N$
this implies that such a walk cannot stay on a finite segment
indefinitely - it will a.s. visit every vertex to the right of its current location. Therefore, when trying to determine whether the process is recurrent or transient, we can
assume that the walk starts at any vertex of $\N$, as long as the
conditions of the theorem still hold for the RWCE at that time. Ellipticity also means that the walk is recurrent (by our definition) exactly when it visits some vertex infinitely many times almost surely and transient exactly when it visits some vertex only finitely many times almost surely.
Throughout this section we write $C(j)$ instead of $C(j,j+1)$ to abbreviate notation.

\begin{theorem} \label{inc_rec_N}
If $\{(X_t,G_t)\}$ is a monotone increasing adaptive RWCE on
$\N$, bounded above by some recurrent connected graph
$G_\infty=(\N,C_\infty)$ then the walk is recurrent.
\end{theorem}

\begin{proof}
Notice that since the RWCE may be adaptive, $G_\infty$ is just a
bound on $G_i$ and not necessarily its limit.

Assume that the walk starts at some $X_0>0$. We will show that the walk almost surely hits $0$. Since the conditions of the Theorem continue to hold at this hitting time, this implies the walk will a.s. hit $0$ infinitely often and is therefore recurrent.
The potential sequence we use in this case is
$$F_t(v)=\sum_{j=0}^{v-1} \frac{1}{C_t(j)}$$
i.e. the resistance between $0$ and $v$ on the graph $G_t$. That
$F_t$ is harmonic on $G_t$ is well known (and easily verified).
Monotonicity follows from the monotonicity of the RWCE. Therefore,
$F_t(X_t)$ is a super-martingale until the first time $X_t=0$.

Since $G_\infty$ is recurrent, we know that $\sum_{j=0}^\infty
1/C_\infty(j) = \infty$. Therefore, given any $A>0$ there is a $v\in
\N$ such that $F_\infty(v) \geq A$. Let $\tau$ to be the first time
the walk hits either $v$ or $0$. By ellipticity, $\tau$ is finite almost surely.
By the optional stopping theorem $F_0(X_0) \geq \E(F_\tau(X_\tau))$.
Denote by $p_v$ the probability that $X_\tau=v$, i.e. that the RW hits
$v$ before 0. Noting that $F_t(0)=0$ for all $t$ and that
$F_t(v)\geq F_\infty(v) \geq A$ we have
$$F_0(X_0) \geq E(F_\tau(X_\tau)) \geq A p_v$$
and therefore
$$p_v \leq \frac{F_0(X_0)}{A} \ .$$

Since $A$ was arbitrary, the proof is complete.
\end{proof}

\begin{theorem} \label{inc_tra_N}
If $\{(X_t,G_t)\}$ is a monotone increasing adaptive RWCE on
$\N$, with $G_0$ transient, and bounded above by some transient graph
$G_\infty=(\N,C_\infty)$ then the walk is transient.
\end{theorem}

\begin{proof}
Note that $G_0$ bounds the sequence $G_t$ from below. The potential
sequence is
$$F_t(v)=\sum_{j=v}^\infty \frac{1}{C_t(j)}$$
i.e. the resistance between $v$ and infinity. Harmonicity and
monotonicity hold as above and $F_t(X_t)$ is therefore a super-martingale. $G_0$ is transient, thus, given
$\epsilon >0$ there is $v\in \N$ such that $F_0(v)<\epsilon$. By ellipticity, we may
assume that $X_0=v$ and since $F_t(v)$ is decreasing we have $F_0(X_0)<\epsilon$.

Let $\tau$ be the first time $X_t=0$, or infinity if the walk never
reaches 0. Let $p_v$ be the probability that $\tau<\infty$. Since $F_t(X_t)$ is positive and using the
optional stopping theorem we have
$$p_v F_\infty(0)\leq \E(F_\tau(X_\tau)) \leq F_0(X_0)$$
which implies
$$p_v \leq \frac{\epsilon}{F_\infty(0)} \ .$$

Since $\epsilon$ was arbitrary, there exists a vertex $v\in \N$ such
that $p_v<\frac{1}{2}$. Ellipticity implies that whenever the walk
is at 0 it will almost surely visit $v$ at some later time and
thereafter it would never visit 0 again with probability
$\frac{1}{2}$.  Therefore, 0 would be visited only a finite number
of times, almost surely.
\end{proof}

\begin{theorem} \label{dec_tra_N}
If $\{(X_n,G_n)\}$ is a monotone decreasing adaptive RWCE on
$\N$, bounded below by some transient graph $G_\infty=(\N,C_\infty)$,
then the walk is transient.
\end{theorem}

\begin{proof}
The potential sequence is
$$F_t(v)=\sum_{j=0}^{v-1} \frac{1}{C_t(j)}$$
i.e. the resistance between 0 and $v$. Harmonicity and
monotonicity hold as above and $F_t(X_t)$ is therefore a sub-martingale.
Obviously, this sub-martingale is bounded by 0 and
$F_\infty(\infty)=\sum_{j=0}^\infty 1/C_\infty(j)$ which is finite.

Assume that the walk starts at $X_0>0$ and fix some $v>X_0$.
Let $\tau$ be the first time the walk hits 0 or $v$, which, by ellipticity, happens almost surely.
Let $p_v$ be the probability that the walk hits $v$ first.
By the optional stopping theorem we have
$$F_0(X_0) \leq E(F_\tau(X_\tau)) \leq (1-p_v) F_\infty(0) + p_v F_\infty(v) \leq p_v F_\infty(\infty)$$
and therefore
$$p_v \geq \frac{F_0(X_0)}{F_\infty(\infty)} \ .$$

This holds for all $v > X_0$ and thus the probability that the walk
never visits 0 is at least $F_0(X_0)/F_\infty(\infty)$. Since $F_i$ is increasing,
this bound holds every time the walk returns to $X_0$ and therefore the walk will
visit 0 only finitely many times, almost surely.
\end{proof}

\begin{theorem} \label{dec_rec_N}
If $\{(X_t,G_t)\}$ is a monotone decreasing adaptive RWCE on
$\N$, with $G_0$ recurrent and bounded below by $G_\infty=(\N,C_\infty)$ with $C_\infty = c\ C_0$ for some $0<c$, then the walk is
recurrent.
\end{theorem}

\begin{proof}
Let $X_0$ be arbitrary. Given $X_0$, let $n$ be such that
\begin{equation} \label{dec_rec_N:requiremnt}
\frac{1}{2} \sum_{j=0}^{n-1} \frac{1}{C_0(j)} \le \sum_{j=X_0}^{n-1}
\frac{1}{C_0(j)} \ .
\end{equation}

This is possible since $G_0$ is recurrent. The potential sequence
will be
$$F_t(v) = \sum_{j=v}^{n-1} \frac{1}{C_t(j)} $$
i.e. the resistance between $v$ and $n$. Then $F_t(X_t)$ is a
sub-martingale until the first time that the RW reaches either 0 or
$n$. Let $\tau$ be that time and let $p_0$ be the probability that
$X_\tau=0$. By the optional stopping theorem we have
$$F_0(X_0) \leq \E(F_\tau(X_\tau)) = p_0 \E(F_\tau(0) | X_\tau =0) +
(1-p_0) 0 \leq p_0 F_\infty(0) = \frac{p_0 F_0(0)}{c}$$

Combining the above with \eqref{dec_rec_N:requiremnt} we conclude that $p\ge
c/2$. This bound holds for any $X_0$, i.e. regardless of the current
state of the RWCE, the probability of reaching 0 in the future is at
least $c/2$. A standard argument
now shows that this probability must actually be 1.
\end{proof}

Unlike the other theorems in this section, the last theorem requires
the RWCE to have bounded ratio between $G_0$ and $G_\infty$. As the
example below shows, this requirement is essential.

\begin{example} \label{dec_rec_N_ex1}
The RWCE with conductances $C_t(j)=2^{-j}$ for $j<t$ and $C_t(j)=1$
otherwise is monotone decreasing, nonadaptive, bounded from above and below by a recurrent graph and is of mixed type.
\end{example}
\begin{proof}
Indeed, with probability $\prod_{t=0}^\infty 1/(1+2^{-t}) >0$ the
RW will always go to the right and otherwise it will eventually perform a
simple random walk on the graph with conductances $2^{-t}$, which
is recurrent.
\end{proof}

It is not too difficult to make this example transient.
Let $D^n_t(j)=2^{-j}$ for $n \le j<t$ and 1 otherwise. So the
conductances of the last example are $D^0_t(j)$.

\begin{example} \label{dec_rec_N_ex2}
There exists an increasing sequence $t_n$ such that the RWCE with conductances $C_t(j)=\prod_{\{n | t_n<t\}} D^n_{t-t_n} (j)$
is monotone decreasing, nonadaptive, bounded from above and below by a recurrent graph and transient.
\end{example}
\begin{proof}[Sketch of proof]
In example \ref{dec_rec_N_ex1} we had, in essence, a "wave" of conductances
$2^{-j}$ threatening to carry the RW away. In this example there's a
multitude of such waves, each starting one edge further, so that the
final conductance of each edge is finite, and each have some fixed
positive probability of carrying the RW away. The sequence $t_n$ is chosen to be increasing fast enough, so that the RW would have a fixed
positive probability of being to the right of the $n$-th "wave" when
it starts.
\end{proof}

\section{RWCE on trees} \label{main_T}

Theorems \ref{inc_rec_N} and \ref{dec_tra_N} can be extended to the case where the underlying graph
is a tree. In order to do that first notice that both proofs use the
same potential sequence. Second, notice that these functions can be
described as follows: Consider the trivial unit flow (on $\N$) from 0 to
infinity and fix the potential at 0 to be 0. Then $F_t(v)$ is
the potential of $v$ in $G_t$. If the underlying graph is a tree,
there are many possible choices of flows, each determining a
potential. Harmonicity and monotonicity are true for any of these
potential sequences, but some care in choosing the right flow is
still needed.
For general graphs, however, this method fails, since not every flow
determines a potential. More precisely, if the graph contains
cycles, then there are 2 distinct flows from the source to some
vertex and the potential is well defined only when Kirchoff's cycle law is satisfied,
which is not necessarily the case.

\begin{theorem} \label{inc_rec_T}
If $\{(X_t,G_t)\}$ is a monotone increasing adaptive RWCE on a tree $T$,
bounded above by some recurrent tree $G_\infty=(T,C_\infty)$ then the walk is recurrent.
\end{theorem}
\begin{proof}
Fix $A>0$. Since $G_\infty$ is recurrent, there is some $n\in \N$
such that the effective resistance (in $G_\infty$) between the root of the tree
(denoted 0) and the outside of the ball of radius $n$ around 0 is at
least $A$, i.e.
$$R=\reff (0 \leftrightarrow \partial B_n(0) ; G_\infty) \ge A \ .$$

Fix such an $n$ and let $i$ be the unit current flow induced by putting a
voltage difference of $R$ between 0 and $\partial B_n(0)$ in $G_\infty$.
Let
$$F_t(v)=\sum_e \frac{i(e)}{C_t(e)}$$
where the sum is over all edges $e$ on the (unique) path connecting 0 and
$v$. In words, $F_t(v)$ is the voltage which is induced by the flow
$i$ on $G_t$. Harmonicity follows, as usual, from Kirchhoff's law
and monotonicity is trivial since $F_t(v)$ is a fixed positive linear
combination of $C_t(e)$'s. Therefore, $F_t(X_t)$ is a
super-martingale until the first time $X_t=0$ or $X_t\in \partial B_n(0)$.
From the definition of the flow $F_\infty(v)=R$ for any $v\in
\partial B_n(0)$. Since $C_t \le C_\infty$ we have
$F_t(v)\ge R \ge A$ for all $v\in \partial B_n(0)$.

The rest of the proof is the same as in theorem \ref{inc_rec_N}.
Let $\tau$ be the first time the walk hits either 0 or $\partial
B_n(0)$. Denote by $p$ the probability that the RW hits $\partial
B_n(0)$ first. Since $F_t(0)=0$ for all $t$, by the optional stopping theorem we have
$$F_0(X_0)\ge \E(F_\tau(X_\tau)) \ge Ap$$
and therefore
$$p \le \frac{F_0(X_0)}{A} \ .$$

Since $A$ was arbitrary, the proof is complete.
\end{proof}

\begin{theorem} \label{dec_tra_T}
If $\{(X_t,G_t)\}$ is a monotone decreasing adaptive RWCE on a tree $T$,
bounded below by some transient tree $G_\infty=(T,C_\infty)$, then the walk is transient.
\end{theorem}

\begin{proof}
To prove transience, it is enough to show that under these conditions there is a vertex $u$
such that such the RWCE, starting from $X_0=u$, has at least some fixed probability of
never returning to 0. Indeed, by ellipticity, every time
the walk returns to 0 it visits $u$ with some fixed probability and
will therefore return to 0 only finitely many time, almost surely.

Since $G_\infty$ transient, the effective resistance, $R$ between between 0 and
infinity is finite, that is, if $i$ is the unit current flow from 0 to infinity then the corresponding potential is bounded by $R$.
Let
$$F_t(v)=\sum_e \frac{i(e)}{C_t(e)}$$
where the sum is over all edges $e$ on the (unique) path connecting 0 and
$v$. This is the same as the previous proof except now $F_t(X_t)$ is
a sub-martingale since $C_t$ is decreasing.

Let $u$ be a neighbor of 0 such that the flow from 0 to $u$ is positive and
assume that $X_0=u$. By definition, $F_0(u)=i(0,u)/C_0(0,u)$ is positive too.

The rest of the proof is the same as in theorem \ref{dec_tra_N}.
Let $\tau$ be the first time the walk hits either 0 or $\partial
B_n(0)$. Denote by $p_n$ the probability that the RW hits $\partial
B_n(0)$ first. By the optional stopping theorem we have
$$F_0(X_0) \leq \E(F_\tau(X_\tau)) = p_n \E(F_\tau(X_\tau) | X_\tau \neq 0) \leq p_n R$$
and therefore
$$p_n \geq \frac{F_0(X_0)}{R} \ .$$

This holds for all $n$ and thus the probability that the walk
never visits 0 is at least $F_0(X_0)/R$. Since $F_t$ is increasing,
this bound holds every time the walk returns to $u$ and therefore the walk will
visit 0 only finitely many times, almost surely.
\end{proof}

\section{RWCE on $\Z^2$} \label{counterexample}

One could hope that the conclusions of Theorem \ref{inc_rec_T} would hold for any monotone
RWCE, but unfortunately, this is not true, as the following
example of an adaptive RWCE in 2 dimensions shows.

The example we build is a monotone increasing adaptive RWCE on
$\Z^2$, with $C_0 = 1$  for horizontal edges and $C_0=2$ for vertical edges and  $C_n \le 2$ for all $n$ and all edges.

We shall try to mimic the behavior of excited
random walk in our model as follows. When the walk reaches a
vertex, we will try to give it a push to the right by increasing
the conductance of the right edge to $2$. If its left neighbor was
never visited, this will make the transition probabilities for the next step
equal to $\frac{1}{7}$ for the left and $\frac{2}{7}$ for all the
other directions which will give a drift to the right. If the left neighbor has already been visited, all transition probabilities will be $\frac{1}{4}$ so the drift
would be zero. Call this walk MAW for {}``Monotone Adaptive Walk''
(the MAW is a specific example of a RWCE).

This is obviously quite
similar to excited random walk so one is tempted to assume we will
get transience, as in \cite{BENJAMINI2003EXCITED}. One should be careful, though,
because in one dimension a similar attempt to mimic the results
of \cite{zerner2005multi} would fail, as Theorem \ref{inc_rec_N} shows. So
this kind of result is quite sensitive to specific details of the
model.

\begin{theorem} \label{thm:MAW}
MAW is transient.
\end{theorem}

We present two proofs for Theorem \ref{thm:MAW}. The first is based on a theorem of Meshnikov and Popov (\cite{menshikov2014range}) on generalized excited random walks, while the second proof is based on the methods of Benjamini and Wilson (\cite{BENJAMINI2003EXCITED}) and on harmonic analysis of random walks in $2$ dimensions. The second proof appears in the appendix to this paper. Part of the reason for keeping the second proof is that the methods in it were used and cited in other works (e.g. \cite{dembo2014monotone}).

\begin{proof}
To show that MAW is transient, we use a result of Menshikov and Popov (\cite{menshikov2014range}, Theorem 1.4) which gives bounds on the size of the range of so called "strongly directed" submartingales. We will show that MAW is such a strongly directed submartingale. Indeed, following definition 1.1 of \cite{menshikov2014range}, we choose $\mathcal{L}$ to be $\R^2$ so that the projection operator $P_\mathcal{L}$ is the identity. $\ell$ is the unit vector in the positive $x$-direction and $u=1$. Then, $H^u_{\ell, \mathcal{L}}$ (in the paragraph above definition 1.1 of \cite{menshikov2014range}) is simply the positive $x$-axis and we see that a submartingale is $(u,\ell, \mathcal{L})$-strongly directed if the drift at any step (conditioned on the history) is in the positive $x$-direction (including 0). Hence, MAW is $(u,\ell, \mathcal{L})$-strongly directed. It is also uniformly elliptic and has uniformly bounded jumps, thus satisfying the conditions of Theorem 1.4 of \cite{menshikov2014range}.

The conclusion of Theorem 1.4 of \cite{menshikov2014range} is that there are constants $\gamma<\frac12$ and $C_1,C_2,\delta>0$ such that
\begin{equation}
\P(R_n < n^{1-\gamma})\le C_1 n e^{-C_2 n^\delta}, \label{eq:R}
\end{equation}
where $R_n$ denotes the number of distinct vertices visited up to time $n$. Define $S_{i,j}$ to be the $3\times 3$ square centered at $(i,j)$, that is, $S_{i,j}=\{i-1,i,i+1\}\times\{j-1,j,j+1\}$. Consider all such squares centered at multiples of 3 - they are all disjoint. Denote by $R'_n$ the number of such squares that are visited by the random walk until time $n$. Since each such square has only 9 distinct vertices, it follows that $R'_n\ge R_n/9$.

Let $\tau_{i,j}$ be the first hitting time of $S_{i,j}$. Call a square \emph{good} if, during the 4 steps immediately after $\tau_{i,j}$, the walk stays in $S_{i,j}$ and visits $(i+1,j)$, but not $(i,j)$. It is straightforward to check that no matter where the walk enters $S_{i,j}$, it can hit $(i+1,j)$ in 4 steps, without leaving the square or hitting $(i,j)$ and that each of those 4 steps has transition probability at least $\frac14$. Thus, conditioned on the history up to time $\tau_{i,j}$, the probability that the square $S_{i,j}$ is good is at least $\frac1{256}$.

Let $Q_k$ be the number of good squares among the first $k$ squares visited (that are centered at multiples of 3). $Q_k$ stochastically dominates a $Binom(k,1/256)$ r.v. hence 
the probability that $Q_k<k/1000$ decays exponentially in $k$. Combining with \eqref{eq:R} we deduce that there exist some $C_3,C_4,\delta_1>0$ such that
\begin{equation}
\P(R''_n<n^{1-\gamma}/1000)\le C_3 e^{-C_4 n^{\delta_1}}, \label{eq:R''}
\end{equation}
where $R''_n$ is the number of good squares visited by time $n$.

Let $R'''_n$ be the number of times up to time $n$ that the MAW visits a vertex before visiting its left neighbour. These are exactly the times the MAW takes a step to the right with probability $2/7$ and to the left with probability $1/7$, conditioned on the history until that time. At all other times, the MAW's step is balanced. Note that every time the MAW visits a good square $S_{i,j}$ and gets to $(i+1,j)$ for the first time, its left neighbour has not been visited, thus $R'''_n\ge R''_{n-4}$.

Let $X_n$ be the $x$-coordinate of the MAW. Applying Azuma's inequality to the martingale $X_n-R'''_n/7$ we get that there exist some $C_5,C_6,\delta_2>0$ such that
\begin{equation}
\P(|X_n-R'''_n/7|<n^{1-\gamma}/14000)\le C_5 e^{-C_6 n^{\delta_2}}. \label{eq:X}
 \end{equation}
Combining \eqref{eq:X} with \eqref{eq:R''} we deduce that there exist some $C_7,C_8,\delta_3>0$ such that
\[
\P(X_n < n^{1-\gamma}/14000)\le C_7 e^{-C_8 n^{\delta_3}}.
\]

Summing for all values of $n$ and using the Borel-Cantelli lemma we conclude that $\lim_{n\to\infty} X_n = \infty$ almost surely and in particular the MAW is transient.
\end{proof}

\section{A conjecture and some open problems} \label{sec:open}
The following conjecture seems the most interesting to us:

\begin{conjecture}
Theorems \ref{inc_rec_N}, \ref{inc_tra_N}, \ref{dec_tra_N} and \ref{dec_rec_N} are
true for nonadaptive RWCE on any graph.
\end{conjecture}

\begin{remark}
Note that these conjectures claim that there is an essential
difference between adaptive and nonadaptive walks. The Theorems in this paper do not provide proof of any such difference. However, as pointed to us by Ben Morris, it seems that using the methods of evolving sets (\cite{morris2005evolving}) it is possible to show that when the graph satisfies an isoperimetric inequality that implies transience (e.g. $\Z^d$ for $d\ge 3$) and the environment is monotone, bounded between two constant multiples of the same conductance function and nonadaptive the RWCE is also transient.
On the other hand, an adaptive example similar that in Section
\ref{counterexample} can likely be constructed on $\Z^d$ by mimicking
the behaviour of excited random walk towards the middle -  a walk
which gets a bias towards the origin every time it visits a new
vertex, for which there is a sketch of proof for recurrence \cite{enteratyourownrisk}. It seems , however, that proving recurrence of the ``excited towards the middle'' RWCE is more technically involved than Theorem \ref{thm:MAW}.
\end{remark}

One could also consider a continuous time version of the RWCE, where the edges are equipped with Poisson clocks with rates equal to their conductance,
and the walk jumps over whatever edge rings first.
There seems to be an essential difference between the continuous time and discrete time models,
which is that for a continuous time random walk on a graph, the stationary measure is always uniform. In fact, it follows from results of  Delmotte and Deuschel \cite{DD} that for continuous time \textbf{nonadaptive} RWCE on $\Z^d$ with conductances bounded above and below by a constant, heat kernel
behavior is essentially the same as in $\Z^d$, and in particular the walk is recurrent if and only if $d\leq 2$.
However, such heat kernel estimates, and even questions of recurrence vs. transience, are open for more general graphs (even when requiring monotonicity).
We therefore ask:
\begin{question}
Do Theorems \ref{inc_rec_N}, \ref{inc_tra_N}, \ref{dec_tra_N} and \ref{dec_rec_N} hold for continuous time nonadaptive RWCE on any graph?
\end{question}

In fact, one could ask the same thing in the discrete time RWCE, by simply requiring all the stationary measures on the graphs $G_i$ to be the same (e.g. by keeping the sum of the conductances at each vertex fixed).
\begin{question}
Do Theorems \ref{inc_rec_N}, \ref{inc_tra_N}, \ref{dec_tra_N} and
\ref{dec_rec_N} hold for nonadaptive RWCE on any graph under the condition that all $G_i$ share the same stationary measure?
\end{question}
The latter question seems closely related to the results of \cite{avin2008explore} on RWCE's on finite graphs, where it is shown that contrary to the general case where the cover time may be exponential, for a  sequence of graphs with common stationary measures, the cover time is only polynomial.

Note that in the adaptive case, the ``fixed'' stationary measure plays no role, and in fact one can easily mimic the behavior of discrete time adaptive RWCE with bounded conductances
using adaptive continuous RWCE's up to a time change --- simply use the same graph, conductances and adaptive rule and change the environment immediately after the process jumps. More precisely, denoting by $(Y_i,D_i)$ the discrete RWCE and by $(X_t,C_t)$ the continuous time RWCE on $G$, and by $T_i$ the time of the $i$'th jump of $X_t$ we get a coupling of the two processes up to time change simply by taking
$(Y_i,D_i)=(X_{T_i},C_{T_i})$.

\section*{Acknowledgments}
The first draft of this paper was written in 2006, and some version
circulated. We wish to thank all those who read and commented on
earlier drafts. We thank Amir Dembo and Ruojun Huang for helpful observations regrading continuous time RWCE and for pointing out the reference \cite{DD}. We thank the anonymous referees for pointing out \cite{menshikov2014range} and its relevance to Theorem \ref{thm:MAW}, as well as other useful suggestions.

The research of G.A.\ was supported by the Israel Science Foundation grant ISF 1471/11 by a Grant
from the GIF, the German-Israeli Foundation for Scientific Research and Development.
\small
\bibliographystyle{plain}

\bibliography{RWCE}

\section*{Appendix: second proof of Theorem \ref{thm:MAW}}

From a {}``calculatory'' point of view, \cite{BENJAMINI2003EXCITED} reduces to the
fact that simple random walk starting from $(0,0)$ has a
probability of $n^{-1/4}$ to avoid hitting the right half line for
the first $n$ steps (Kesten's lemma: see \cite{kesten1987hitting}). This
$n^{-1/4}$ factor manifests itself in the fact that, finally, they
prove a $n^{3/4}$ drift (up to logarithmic factors). In their
settings the probabilities for going up or down never change ---
the effect of the drift is only to move weight around between the
left and right probabilities. We do not know how to mimic this
particular detail in the settings of monotone adaptive
conductances so we will need to work without it, and this would
complicate the geometric settings somewhat. Below we work out our
replacement for Kesten's lemma. Hence for a while we will only
develop properties of simple random walk. The impatient can jump
to lemma \ref{lem:tan nonNV} to see how this is used. Below
$\epsilon\in(0,\frac{1}{2})$ is some parameter that will be kept
fixed throughout. The notation $x\approx y$ denotes that $cX \leq y \leq Cx$ for some absolute constants $c,C>0$.

\begin{lem}
Let $E_{1}$ be the event that a random walk starting from $0$ will
avoid hitting the point $(-1,0)$ for the first $\left\lceil
n^{\epsilon}\right\rceil $ steps. Then \[
\mathbb{P}(E_{1})\approx\frac{1}{\log n}.\]

\end{lem}
This is a well known fact. See e.g.~\cite{spitzer2001principles}.

\begin{lem}
\label{lem:Wtlog}Let $W$ be Brownian motion starting from $0$. Let
$K$ be an infinite cone with opening $\theta\in[0,\pi]$ and tip
$v$, and assume $|v|\leq2d(0,K)$. Let $t>4|v|^{2}$ and denote
$\mu=\sqrt{t}/|v|$. Then\begin{align}
\mathbb{P}\left(W[0,t]\cap K=\emptyset\right) & \geq\frac{c\mu^{-\pi/(2\pi-\theta)}}{\sqrt{\log\mu}}\label{eq:WKbig}\\
\mathbb{P}(\{ W[0,t]\cap
K=\emptyset\}\cap|W(t)_{2}|<\delta\sqrt{t}) & \leq
C\delta\mu^{-\pi(2\pi-\theta)}\sqrt{\log\mu}.\label{eq:WKdsmall}\end{align}
for any $\delta<1$.
\end{lem}
We remark that both $\sqrt{\log}$ factors above can be removed
without much difficulty. See some additional blurbs on this in the
remark on page \pageref{rem:2nd mom} below.

\begin{proof}
Denote $\xi=\pi/(2\pi-\theta)$. Let $T_{r}$ be the stopping time
of $W$ on $\partial B(0,r)$. Applying the map $z\rightarrow z^{\xi}$ that maps the cone to a half-space , and using conformal invariance (and a few
calculations) we get that \[ \mathbb{P}(W[0,T_{r}]\cap
K=\emptyset)\approx\left(\frac{|v|}{r}\right)^{\xi}.\] On the
other hand, $\mathbb{P}(t>T_{C\sqrt{t\log\mu}})\leq\mu^{-1}$ for
some $C$ sufficiently large, so\[ \mathbb{P}(W[0,t]\cap
K=\emptyset)\geq\mathbb{P}(W[0,T_{C\sqrt{t\log\mu}}]\cap
K=\emptyset)-\mu^{-1}\geq
c\left(\frac{1}{\mu\sqrt{\log\mu}}\right)^{\xi}.\] For the other
part, first notice that
$\mathbb{P}(\frac{1}{2}t<T_{c\sqrt{t/\log\mu}})\leq\mu^{-1}$ which
gives similarly that $\mathbb{P}(W[0,\frac{1}{2}t]\cap
K=\emptyset)\leq c(\sqrt{\log\mu}/\mu)^{\xi}$. After
$\frac{1}{2}t$ we have that for any $x\in\mathbb{R}^{2}$ that
Brownian motion $W$ starting from $x$ has probability $\leq
C\delta$ to be in the strip $\{(x,y):|y|<\delta\sqrt{t}\}$. These
two facts prove (\ref{eq:WKdsmall}).
\end{proof}
\begin{lem}
\label{lem:E2}For $m\in(n^{2\epsilon},n]$ let $E_{2}(m,n)$ be the
event that a random walk $R$ starting from $0$, satisfies
\begin{enumerate}
\item $E_{1}$ \item $R\left[\left\lceil
n^{\epsilon}\right\rceil ,m\right]\cap F=\emptyset$ where $F$ is
the funnel \begin{equation}
F=\{(x,y):x\geq-1,|y|\leq\log^{3}n\sqrt{x+2}\}.\label{eq:funnel}\end{equation}

\end{enumerate}
Then \begin{align*}
\mathbb{P}\left(E_{2}(m,n)\right) & \geq\left(\frac{n^{\epsilon}}{m}\right)^{1/4+o(1)},\\
\mathbb{P}\left(E_{2}(m,n)\cap\{\left|R(m)_{2}\right|<\delta\sqrt{m}\}\right)
&
\leq\delta\left(\frac{n^{\epsilon}}{m}\right)^{1/4+o(1)}\end{align*}
if only $\delta\sqrt{m}\geq1$.
\end{lem}
Here and below $o(1)$ stands for an entry that goes to $0$ as
$n\to\infty$ uniformly in $m>n^{2\epsilon}$.

\begin{proof}
Denote $v=R(\left\lceil n^{\epsilon}\right\rceil )$. The first
ingredient is Hungarian coupling \cite[Theorem 4]{einmahl1989extensions}, see also
\cite{chatterjee2012new,zaitsev1998multidimensional,komlos1976approximation}, which gives that we can couple random
walk starting from $v$ to Brownian motion $W$ also starting from
$v$ such that with probability $\ge 1 - n^{-10}$ we have $|R(t)-W(t)|\leq
C_{1}\log^{2}t$. We therefore find two cones $K^{\pm}$ satisfying
$K^{-}+B(0,C_{1}\log^{2}n)\subset F$ and
$F+B(0,C_{1}\brmul\log^{2}n)\subset K^{+}$. Specifically we
choose\begin{align*}
K^{-} & =\left\{ (x,0):x\geq C_{1}\log^{2}n\right\} \\
K^{+} & =\left\{
(x,y):x\geq-n^{\epsilon/4},|y|\leq\frac{\log^{3}n}{n^{\epsilon/8}}(x+n^{\epsilon/4})\right\}
\end{align*} and the inclusion conditions will be satisfied for
$n$ sufficiently large.

Next we want to estimate the distance of $v$ from $K^{+}$. With
probability $>1-C\log^{-2}n$ we have that
$d(v,K^{+})>n^{\epsilon/2}\log^{-2}n$. To see this fix some
$\lambda=1,2,\dotsc$ and examine the annulus
$A:=n^{\epsilon/2}(B(0,\lambda)\setminus B(0,\lambda-1))$. For
every $w\in A$ one has that $\mathbb{P}(v=w)\leq
Cn^{-\epsilon}e^{-\lambda^{2}}$ while the inflated cone
$\big(K^{+}+B(0,n^{\epsilon/2}\log^{-2}n)\big)\cap A$ contains
$\leq Cn^{\epsilon/2}\big(\lambda
n^{(3/8)\epsilon}\log^{2}n+n^{\epsilon/2}\log^{-2}n\big)\leq
C\lambda n^{\epsilon}\log^{-2}n$ points. Summing over $\lambda$ we
get the estimate for $d(v,K^{+})$. Comparing to the probability of
$E_{1}$ we get for $n$ sufficiently large\begin{equation}
\mathbb{P}(E_{1}\cap\{
d(v,K^{+})>n^{\epsilon/2}\log^{-2}n\})\approx\frac{C}{\log
n}.\label{eq:E1dvF2}\end{equation} Now we may invoke lemma
\ref{lem:Wtlog} and get that, assuming
$d(v,K^{+})>n^{\epsilon/2}\log^{-2}n$,\begin{multline*}
\mathbb{P}(R[\left\lceil n^{\epsilon}\right\rceil ,m]\cap F=\emptyset)\geq\mathbb{P}(W[0,m-\left\lceil n^{\epsilon}\right\rceil ]\cap K^{-}=\emptyset)\\
\stackrel{(\ref{eq:WKbig})}{\geq}\frac{c}{\sqrt{\log
n}}\left(\frac{n^{\epsilon/2}\log^{-2}n}{\sqrt{m-\left\lceil
n^{\epsilon}\right\rceil
}}\right)^{1/4}\geq\left(\frac{n^{\epsilon}}{m}\right)^{1/4+o(1)}\end{multline*}
and\begin{eqnarray*}
\lefteqn{{\mathbb{P}\left(\left\{ R\left[\left\lceil n^{\epsilon}\right\rceil ,m\right]\cap F=\emptyset\right\} \cap\left\{ |R(m)_{2}|\leq\delta\sqrt{m}\right\} \right)\leq}}\\
 & \qquad & \leq\mathbb{P}\Big(\left\{ W\left[0,m-\left\lceil n^{\epsilon}\right\rceil \right]\cap K^{+}=\emptyset\right\} \cap\\
 &  & \qquad\qquad\cap\left\{ \left|W(m-\left\lceil n^{\epsilon}\right\rceil )_{2}\right|\leq\delta\sqrt{m}+C_{1}\log^{2}n\right\} \Big)\leq\\
 &  & \stackrel{(\ref{eq:WKdsmall})}{\leq}C(\delta+\frac{C_{1}\log^{2}n}{\sqrt{m}})\sqrt{\log n}\left(\frac{n^{\epsilon/2}\log^{-2}n}{\sqrt{m-\left\lceil n^{\epsilon}\right\rceil }}\right)^{\pi/(2\pi-n^{-\epsilon/8}\log^{3}n)}\leq\\
 &  & \leq\delta\left(\frac{n^{\epsilon}}{m}\right)^{1/4+o(1)}\end{eqnarray*}
Where in the last inequality we used $\delta\sqrt{m}\geq1$ to
bound $C_{1}m^{-1/2}\log^{2}n\leq\delta\log^{2}n$ and then this
$\log$ factors can be folded into the $o(1)$ in the exponent like
all the other $\log$-s (including the one from $E_{1}$). Notice
also that we didn't write the negligible probability for the
coupling to fail, but it does not affect the result for $n$
sufficiently large.
\end{proof}
\begin{lem}
\label{lem:E3}Let $E_{3}(m,n)$, $m\geq n^{2\epsilon}$ be the event
that a random walk $R$ starting from $0$ satisfies that
\begin{enumerate}
\item \label{enu:tan1}$R\big[m-\left\lceil
n^{\epsilon}\right\rceil ,m\big]$ avoids $(-1,0)+R(m)$; and \item
\label{enu:tan2}$R\big[0,m-\left\lceil n^{\epsilon}\right\rceil
\big]$ avoids $F+R(m)$ where $F$ is the funnel defined in
(\ref{eq:funnel}).
\end{enumerate}
Then \begin{align}
\mathbb{P}(E_{3}) & \geq\left(\frac{n^{\epsilon}}{m}\right)^{1/4+o(1)},\label{eq:PE3>}\\
\mathbb{P}(E_{3}\cap\{|R(m)_{2}|<\delta\sqrt{m}\}) &
\leq\delta\left(\frac{n^{\epsilon}}{m}\right)^{1/4+o(1)}.\label{eq:PE3del<}\end{align}
 for any $m\in[n^{2\epsilon},n]$.
\end{lem}
\begin{proof}
This follows immediately from lemma \ref{lem:E2} and time reversal
symmetry.
\end{proof}
Following \cite{BENJAMINI2003EXCITED} we will call $m$ satisfying $E_{3}(m,n)$
{}``tan points'' (imagine the sun being at the right infinity,
then $R(m)$ gets a tan without (almost) any previous point
blocking a whole {}``tanning funnel'').

\begin{lem}
\label{lem:E32}Let $m_{1}<m_{2}$ and
$m_{1},m_{2}-m_{1}>n^{2\epsilon}$. Then \[
\mathbb{P}(E_{3}(m_{1},n)\cap
E_{3}(m_{2},n))\leq\mathbb{P}(E_{3}(m_{1},n))\mathbb{P}(E_{3}(m_{2}-m_{1},n)).\]

\end{lem}
\begin{proof}
One only needs to notice that it is easier for $R(m_{2})$
to be a tan point \emph{with respect to the walk starting from
$R(m_{1})$} then to be a regular tan point. In other words, if
$S(i):=R(m_{1}+i)-R(m_{1})$ then $S$ is a random walk starting
from $0$; and if $E^{*}$ is the event that $m_{2}-m_{1}$ is a tan
point for $S$; then $E_{3}(m_{2},n)\subset E^{*}$.
\end{proof}
\begin{lem}
\label{lem:manytans}With probability $>1-Cn^{-2}$ there are at
least $n^{3/4-(7/4)\epsilon+o(1)}$ $n^{\epsilon}$-separated tan
points up to time $n$.
\end{lem}
\begin{proof}
Let $h=\left\lceil n^{1/2}\log^{-2}n\right\rceil $ and
$l=\left\lfloor h^{2}\log^{-1}n/\left\lceil
n^{2\epsilon}\right\rceil \right\rfloor $. Let $T_{i}$ be the
stopping times on the double line $\{(x,\pm ih):x\in\mathbb{R}\}$.
For all $i\in\mathbb{N}$ and $j=l,l+1,\dotsc,2l$ let $Y_{i,j}$ be
the event that $T_{i}+j\left\lceil n^{2\epsilon}\right\rceil $ is
a tan point with respect to $T_{i}$. Define $X_{i}:=\#\{
j:Y_{i,j}\}$. The first step is to show that \begin{equation}
\mathbb{P}(X_{i}>n^{3/4-(7/4)\epsilon+o(1)})>c.\label{eq:Xic}\end{equation}
We use second moment methods. First by (\ref{eq:PE3>}) we have \[
\mathbb{E}(X_{i})\geq(l+1)\cdot\left(\frac{n^{\epsilon}}{2l\left\lceil
n^{2\epsilon}\right\rceil }\right)^{1/4+o(1)}\geq
n^{3/4-(7/4)\epsilon+o(1)}.\] For the second moment write
\begin{align*}
\mathbb{E}\left(X_{i}^{2}\right) & =\sum_{j}\mathbb{P}(Y_{j})+\sum_{j<k}2\mathbb{P}(Y_{j}\cap Y_{k})\\
\intertext{\textrm{and by lemma \ref{lem:E32}}} &
\leq\mathbb{E}X_{i}+\sum_{j<k}2\mathbb{P}(Y_{j})\mathbb{P}(Y_{k-j})\leq\mathbb{E}X_{i}+2\left(\mathbb{E}X_{i}\right)^{2}.\end{align*}
By the well known inequality
$\mathbb{P}(X\geq\frac{1}{2}\mathbb{E}X)\geq(\mathbb{E}X)^{2}/4\mathbb{E}(X^{2})$
we get for $n$ sufficiently large \[
\mathbb{P}(X_{i}>n^{3/4-(7/4)\epsilon+o(1)})\geq\frac{1}{12}.\]

Next we define $Y_{i,j}^{*}$ to be the event\[ Y_{i,j}\cap\left\{
\left|R(T_{i}+j\left\lceil n^{2\epsilon}\right\rceil
)_{2}\right|>ih+n^{1/4}\log^{4}n\right\} .\] And $X_{i}^{*}=\#\{
j:Y_{i,j}^{*}\}$. We shall now estimate $X_{i}^{*}$ under the
assumption that $R(T_{i})_{2}=ih$ (rather than $-ih$)
--- the other case is symmetric. Examine the event \[
B_{i,j}=Y_{i,j}\cap\left\{ \left|R(T_{i}+j\left\lceil
n^{2\epsilon}\right\rceil )_{2}-ih\right|\leq
n^{1/4}\log^{4}n\right\} .\] By (\ref{eq:PE3del<}) we have
(remember the definitions of $l$ and $h$) that\[
\mathbb{P}(B_{i,j})\leq\frac{n^{1/4}\log^{4}n}{\sqrt{jn^{2\epsilon}}}\left(\frac{n^{\epsilon}}{j\left\lceil
n^{2\epsilon}\right\rceil }\right)^{1/4+o(1)}\leq
n^{-1/2+\epsilon/4+o(1)}\] and summing over $j$ we get $\E(\#j: B_{i,j}) \leq
n^{1/2-(7/4)\epsilon+o(1)}$. Estimating with Markov's inequality
we see that the $B_{i,j}$ are negligible and then \[
\mathbb{P}(\#\{ Y_{i,j}\setminus
B_{i,j}\}>n^{3/4-(7/4)\epsilon+o(1)})>c.\] Now $Y_{i,j}\setminus
B_{i,j}$ is equal to $Y_{i,j}^{*}\cup\{$its symmetric image$\}$.
Therefore we get \[ \mathbb{P}(X_{i}^{*}>N)=\mathbb{P}(\#\{
j:Y_{i,j}\setminus(B_{i,j}\cup Y_{i,j}^{*})\}>N)\quad\forall N\]
And hence $\mathbb{P}(X_{i}^{*}>N)\geq\frac{1}{2}\mathbb{P}(\#\{
Y_{i,j}\setminus B_{i,j}\}>2N)$.

Finally we define the event\[ G_{i}=\{
X_{i}^{*}>n^{3/4-(7/4)\epsilon+o(1)}\}\cap\{
T_{i+1}-T_{i}>h^{2}\log^{-1}n\}.\] Then since
$\mathbb{P}(T_{i+1}-T_{i}\leq h^{2}\log^{-1}n)<e^{-c\log^{2}n}$ we
get that $\mathbb{P}(G_{i})>c$.

However, $G_{i}$ is independent of $R(T_{i})$, including of
whether it is in the line $\mathbb{R}\times\{ ih\}$ or
$\mathbb{R}\times\{-ih\}$, since $Y_{i,j}^{*}$ and the rest of the
conditions are invariant with respect to translations in the $x$
direction and reflections through the $x$ axis. Therefore (since
$G_{i}$ does not examine the walk beyond $T_{i+1}$) the $G_{i}$
are independent events. Further, with probability $>1-n^{-2}$ we
have $\max_{m\leq n}|R(m)_{2}|\geq c\sqrt{n/\log n}$ and then
there are at least $c\log^{3/2}n$ different $i$-s for which
$T_{i+1}<n$. This shows that with probability $>1-Cn^{-2}$ at
least one of the $G_{i}$-s occurred. Further, with the same
probability we may also assume \begin{equation} \max_{m\leq
n}|R(m)_{1}|\leq\sqrt{n\log n}.\label{eq:maxRm1}\end{equation}

This finishes the lemma. Indeed, \[ (\ref{eq:maxRm1})\cap
Y_{i,j}^{*}\Rightarrow T_{i}+j\left\lceil
n^{2\epsilon}\right\rceil \textrm{ is a tan point}\] since the
funnel $F+R(T_{i}+j\left\lceil n^{2\epsilon}\right\rceil )$
intersects the band $\{(x,y):|y|\leq ih\}$ only for
$|x|>c\sqrt{n}\log^{2}n$ and $R$ does not go so far. Hence the
$T_{i}+j\left\lceil n^{2\epsilon}\right\rceil $-s for which
$Y_{i,j}^{*}$ occurred are $n^{\epsilon}$-separated tan points and
the lemma is proved.
\end{proof}
\begin{remark}
\label{rem:2nd mom}Lemma \ref{lem:E32} alleviates most of the
agony usually associated with second moment methods. However it is
by no means necessary. There are at least two additional paths one
might take to prove the result i.e.~lemma \ref{lem:manytans}:
\begin{itemize}
\item It is not very difficult to get rid of all the $\log$
factors we have so lavishly neglected and show explicitly that
$\mathbb{P}(E_{3})\approx\frac{1}{\log
n}\left(\frac{n^{\epsilon}}{m}\right)^{1/4}$
--- the $1/\log n$ comes from $E_{1}$ and is the only $\log$ that
represents a real phenomenon. Further one can show that
$\mathbb{P}(E_{3}(m,n)\cap\{
R(m)_{2}>\sqrt{m}\})\approx\frac{1}{\log
n}\left(\frac{n^{\epsilon}}{m}\right)^{1/4}$ which would allow to
estimate $X_{i}^{*}$ without going through the symmetry argument.
\item Alternatively, if the second moment methods only show that
$\mathbb{P}(G_{i})>c\log^{-10}n$ one can simply take
$h=\sqrt{n}\log^{-12}n$. This will give $\log^{11+1/2}n$ possible
$i$-s and one of them would satisfy $E_{i}$.
\end{itemize}
\end{remark}
This concludes what we need to know about simple random walk. The
next step is to couple MAW and SRW. We shall do so in the natural
way: if the MAW is in a vertex whose left neighbor was visited in
the past (NV-vertex), make the MAW and the SRW walk together.
Otherwise, do as follows:

\begin{itemize}
\item With probability $\frac{1}{7}$ they both walk to the left
\item With probability $\frac{1}{4}$ they both walk up, another
$\frac{1}{4}$ for right, and another for down. \item With
probability $\frac{1}{28}$ the SRW walks left and the MAW walk up,
etc.
\end{itemize}
Denoting by $R$ the SRW and by $E$ the MAW we get that $E(n)-R(n)$
changes only when $E$ is in a non-NV vertex. $(E(n)-R(n))_{1}$
only increases while $(E(n)-R(n))_{2}$ performs a random walk at
these times.

\begin{lem}
\label{lem:D1D2}Let $R$ and $E$ be an SRW and an MAW coupled as
above. Let $D(k,l)=E(l)-R(l)-E(k)+R(k)$. Let $n$ be some number.
Then\[ \mathbb{P}\Big(\exists k<l<n:|D(k,l)_{2}|\geq C\log
n\sqrt{D(k,l)_{1}+1}\Big)\leq n^{-1}\] for some $C$ sufficiently
large.
\end{lem}
\begin{proof}
Fix some $k<l$. For an integer $K$ let $N(K)$ be the event that exactly $K$ non-NV vertices were visited
by $E$ between $k$ and $l$. It is easy to see that
\begin{equation} \mathbb{P}(N(K),D(k,l)_{1}+1<\frac{1}{\lambda}K)\leq
Ce^{-c\lambda}\quad\forall\lambda>0,\forall
K\label{eq:D1K}\end{equation} so for some $C_{1}$ sufficiently
large, setting $\lambda=C_{1}\log(nK)$ will ensure that the
probability is $\leq\frac{1}{10}(nK)^{-3}$. Denote this event by
$B_{1}(K)$. Next we note that
\begin{equation}
\mathbb{P}(N(K),|D(k,l)_{2}|>\lambda\sqrt{K})\leq
Ce^{-c\lambda^{2}}\quad\forall\lambda>0\label{eq:D2K}
\end{equation}
and setting $\lambda=C_{2}\sqrt{\log(nK)}$ for some $C_{2}$
sufficiently large will ensure that the probability is
$\leq\frac{1}{10}(nK)^{-3}$. Denote this event by $B_{2}(K)$. We
get that
\[ \mathbb{P}\Big(\bigcup_{K=1}^{\infty}N(K)\cup B_{1}(K)\cup
B_{2}(K)\Big)\leq\frac{1}{n^{3}}.\]
However, if this event did not
happen then (\ref{eq:D1K}) gives that the number of non-NV vertices is smaller than $C\log
n(D(k,l)_{1}+1)$ and with (\ref{eq:D2K}) we get $|D(k,l)_{2}|\leq
C\log n\sqrt{D(k,l)_{1}+1}$. Summing over $k$ and $l$ proves the lemma.
\end{proof}
\begin{lem}
\label{lem:tan nonNV}Let $R$ and $E$ be an SRW and an MAW coupled
as above. Assume the event of lemma \ref{lem:D1D2} did not happen.
Let $m<n$ be a tan point of $R$. Then at least one of
$[m-n^{\epsilon},m]$ is a non-NV point of $E$.
\end{lem}
\begin{proof}
If all of $[m-n^{\epsilon},m-1]$ were NV points of $E$, then
$E(l)-R(l)$ did not change throughout this time. Hence the first
condition in the definition of a tan point, that $R(l)$ did not
visit the left neighbor of $R(m)$, ensures that $E(l)$ did not
visit the left neighbor of $E(m)$. So this period is secured.
Examine now the time $[0,m-n^{\epsilon}]$. If for some
$l\in[0,m-n^{\epsilon}]$ we have that $E(l)$ is the left neighbor
of $E(m),$ then $R(l)=\textrm{left neighbor of }R(m)+D(l,m)$. But
we assumed (this is the event of lemma \ref{lem:D1D2}) that
$|D(k,l)_{2}|\leq C\log n\sqrt{D(k,l)_{1}+1}$. Hence (if $n$ is
sufficiently large), $R(l)\in R(m)+F$, in contradiction to the
second condition in the definition of a tan point.
\end{proof}

\begin{proof}[Proof of Theorem \ref{thm:MAW}]
Couple the MAW $E$ to a SRW $R$ as above. Examine the first $n$ steps of
both. By lemma
\ref{lem:manytans} there are (with probability $>1-C/n)$
$n^{3/4+o(1)}$ tan points $m_{i}<n$ which are separated
i.e.~$|m_{i}-m_{j}|>n^{\epsilon}$ for all $i\neq j$. By lemma
\ref{lem:tan nonNV} this shows that there are at least so many
visits of $E$ to non-NV vertices. This shows that with probability
$>1-C/n$ that $(E(n)-R(n))_{1}>n^{3/4+o(1)}$. This shows that with
probability $1$, $E(2^{n})-R(2^{n})>2^{3n/4+o(n)}$. Since
$(E-R)_{1}$ is monotone we get $E(n)-R(n)>n^{3/4+o(1)}$ with
probability $1$. Hence $E$ is transient.
\end{proof}

\end{document}